\documentclass[12pt]{l4dc2021} 

\usepackage{mathrsfs}
\usepackage[font={small}]{caption}
\usepackage{scalerel}
\usepackage{url}
\usepackage{bbold}
\usepackage{amsfonts}
\usepackage{amsmath,amssymb,amsfonts}
\usepackage{graphicx}
\usepackage{mathrsfs} 
\usepackage{algorithm,algorithmic}
\usepackage{array}
\usepackage{times}
\usepackage{url}
\usepackage{upgreek}
\usepackage{float}
\usepackage{longtable}
\usepackage{color}
\usepackage{wasysym}
\usepackage{grffile}
\usepackage{stackengine}
\usepackage{mathtools}
\usepackage{tikz}

\title[~~]{Linear Convergence of Distributed Mirror Descent with Integral Feedback for Strongly Convex Problems}
\usepackage{times}

\newcommand{\0}{\mathbb{0}}
\newcommand{\1}{\mathbb{1}}

\newcommand{\R}{\mathbb{R}}

\newcommand{\bb}{\mathbf{b}}

\newcommand{\rb}{\mathbf{r}}

\newcommand{\wb}{\mathbf{w}}
\newcommand{\xb}{\mathbf{x}}
\newcommand{\yb}{\mathbf{y}}
\newcommand{\zb}{\mathbf{z}}

\newcommand{\Ab}{\mathbf{A}}

\newcommand{\Db}{\mathbf{D}}

\newcommand{\Hb}{\mathbf{H}}

\newcommand{\Lb}{\mathbf{L}}
\newcommand{\Mb}{\mathbf{M}}

\newcommand{\Qb}{\mathbf{Q}}
\newcommand{\Rb}{\mathbf{R}}
\newcommand{\Sb}{\mathbf{S}}

\newcommand{\Wb}{\mathbf{W}}

\newcommand{\Dc}{\mathcal{D}}
\newcommand{\Ec}{\mathcal{E}}

\newcommand{\Gc}{\mathcal{G}}

\newcommand{\Lc}{\mathcal{L}}
\newcommand{\Nc}{\mathcal{N}}

\newcommand{\Vc}{\mathcal{V}}
\newcommand{\Xc}{\mathcal{X}}

\newcommand{\argmin}{\text{argmin}}

\newcommand{\diag}{\text{diag}}
\newcommand{\col}{\text{col}}

\newcommand{\norm}[1]{\left\lVert#1\right\rVert}

\newtheorem{assumption}{Assumption}
\author{%
 \Name{Youbang Sun} \Email{ybsun@tamu.edu}\\
 \Name{Shahin Shahrampour} \Email{shahin@tamu.edu}\\
 \addr {Wm Michael Barnes `64 Department of Industrial \& Systems Engineering, Texas A\&M University}
}
\begin{document}

\maketitle

\begin{abstract}%
  Distributed optimization often requires finding the minimum of a global objective function written as a sum of local functions. A group of agents work collectively to minimize the global function. We study a continuous-time decentralized mirror descent algorithm that uses purely local gradient information to converge to the global optimal solution.  The algorithm enforces consensus among agents using the idea of {\it integral feedback}. Recently, \cite{sun2020distributed} studied the asymptotic convergence of this algorithm for when the global function is strongly convex but local functions are convex. Using control theory tools, in this work, we prove that the algorithm indeed achieves (local) exponential convergence. We also provide a numerical experiment on a real data-set as a validation of the convergence speed of our algorithm.
\end{abstract}


\section{Introduction}

%  Introduce distributed optimization
Distributed gradient-based optimization is well-studied in the literature. Generally, the problem is to find the optimal solution for a global objective function that is a sum of local cost functions assigned to various agents. Each agent only has limited knowledge of the global problem, and the agents must work collectively to reach consensus around the optimum for the global objective function. Distributed optimization has applications in distributed resource allocation \cite{chavez1997challenger}, distributed sensor localization \cite{khan2009distributed}, distributed cooperative control \cite{qu2009cooperativecontrol}, social learning \cite{shahrampour2015distributed}, and beyond. %also good for reasons like reliability \cite{reddi2016stochastic}, scalability \cite{jogalekar2000evaluating} and privacy problems \cite{konevcny2015federated}.

Naturally, one of the most fundamental questions in distributed optimization is that whether a distributed algorithm is able to match the performance of its centralized counterpart. The basic idea of gradient descent with local averaging has proven to be a simple yet powerful approach. The seminal work of \cite{nedic2009distributed} is a prominent point in case, which shows this approach converges for convex problems using a {\it diminishing step-size} sequence, which decreases the influence of local gradients and allows all agents to reach consensus. However, as soon as assumptions like {\it smoothness} and/or {\it strong convexity} come into play, a diminishing step-size may no longer be optimal in centralized optimization, thereby being a sub-optimal choice for decentralized algorithms as well.

%It can be shown that simple gradient descent with local averaging will most likely fail to reach consensus, the agents will each converge to a different point within some locality of the global optimum. To overcome this, many approaches were taken, including diminishing step-sizes \cite{nedic2009distributed}, which decreases the influence of local gradient letting all agents to reach consensus. 
% However, algorithms with diminishing step-sizes will unavoidably converge slower due to its decreasing learning rate. 
A number of works proposed \textit{gradient tracking}, that uses an additional term to ensure consensus with non-decreasing step-sizes. This line of work includes EXTRA \cite{shi2015extra} and DEXTRA \cite{xi2017dextra}, where we can observe decentralized performances on par with their respective centralized problems. %Algorithm for non-convex problems have been discussed in \cite{di2016next}. 
In continuous-time distributed optimization, another approach, termed \textit{integral feedback}, has been used in the literature in a similar spirit. The integral feedback introduces another variable to account for differences between agents and helps the network reach consensus. Examples of recent works adopting this approach include \cite{cortes2013weightbalanceddigraph, kia2015distributed, 7744584, yang2016multi}.

% problems of gradient descent & introduce mirror descent
However, most of the recent works in distributed gradient-based optimization have focused on gradient descent. Although effective, gradient descent sometimes cannot yield desirable results by not exploiting the geometry of the problem. Mirror descent \cite{nemirovsky1983originalmd}, on the other hand, is widely used in large-scale optimization problems. Mirror descent replaces the Euclidean distance in gradient descent with {\it Bregman divergence} as the regularizer, and it can be viewed as a more general version of gradient descent. For some of high-dimensional optimization problems, mirror descent can provide significantly faster convergence rates compared to gradient descent \cite{ben2001ordered}.

% distributed mirror descent works
Motivated by the generality of mirror descent, in this work we focus on distributed mirror descent (DMD). Most of prior work on DMD is in discrete time (see e.g., \cite{shahrampour2017distributed,yuan2018optimal,7383850,li2016distributed,8409957}). With the exception of \cite{7383850},
%In practice, distributed mirror descent has also been applied to social learning and belief dynamics \cite{shahrampour2015distributed}, a comparison of convergence of iterates for both centralized and DMD has been provided by \cite{8409957}. 
the works above either use diminishing step-size sequence or multi-communications per round in order to reach consensus. For the same reasons mentioned for gradient descent, a diminishing step-size would not be optimal for strongly convex problems, resulting in slower convergence compared to centralized methods. In this work, we study continuous-time DMD with integral feedback, recently proposed in \cite{sun2020distributed}. The authors focused on a setup where the global objective is strongly convex but the local functions are convex, and they provided asymptotic convergence analysis. In the current work, we use dynamical systems tools (Lyapunov's indirect method) to prove the {\it local exponential} convergence of DMD with integral feedback. We also test our algorithm on a real data-set to show that the proposed algorithm indeed converges exponentially fast (or linearly in log-scale).

We remark that DMD in continuous time has also been studied prior to this work, mostly by focusing on reduction of noise variance in stochastic optimization \cite{borovykh2020interact,raginsky2012Variancereduction}. \cite{yu2020rlc} also motivate mirror descent using RLC circuits and utilize derivative and integration in the algorithm. The distinction between \cite{yu2020rlc} and the current work includes different assumptions on the objective functions, which yields different convergence results.

% \begin{itemize}
%   \item Limit the main text (not counting references) to 10 PMLR-formatted pages, using this template.
%   \item Include {\em in the main text} enough details, including proof details, to convince the reviewers of the contribution, novelty and significance of the submissions.
% \end{itemize}

\section{Problem Formulation}
{\bf Notation:} We let $[n]$ denote the set $\{1,2,3,\ldots,n\}$ for any integer $n$. $x^\top$ (and $A^\top$) denotes transpose of vector $x$ (and matrix $A$), respectively. $I_d$ represents identity matrix of size $d\times d$. We let $\1_d$ denote $d$-dimensional vector of all ones. $\langle x, y \rangle$ denotes the standard inner product between $x$ and $y$ and $\norm{x} = \sqrt{\langle x, x \rangle}$ is the Euclidean norm of vector $x$. $A\otimes B$ represents the Kronecker product of matrices $A$ and $B$. The $i$-th element of the vector $x$ is denoted by $[x]_i$, and the $ij$-th element of the matrix $A$ is denoted by $[A]_{ij}$. We let $det(A)$ denote the determinant of matrix $A$ and use $\col\{v_1, \ldots, v_n\}$ to denote the vector that stacks all vectors $v_i$ for $i\in [n]$. We use $\diag\{a_1, \ldots, a_n\}$ to represent an $n\times n$ diagonal matrix that has the scalar $a_i$ in its $i$-th diagonal element. We use $Re[\cdot]$ to denote the real part of a complex number. We use $0$ to represent the null vector and the null matrix when it is clear from the context.

\subsection{Distributed Optimization}\label{sec:dist-opt}
Distributed convex optimization consists of minimizing an objective function $F:\Xc \rightarrow \R $ defined on a compact and convex set $\Xc \subset \R^d$. $F$ is written as a sum of local cost functions, denoted by $f_i:\Xc \rightarrow \mathbb{R}$ for $i\in [n]$, and the cost function $f_i$ is associated with agent $i$. The minimization task is as follows
\begin{equation}
   \label{globalfunction}
   \underset{x \in \Xc }{\text{minimize}} \:\:\:\:\:\:F(x) = \sum_{i=1}^n f_i(x).
\end{equation}
In a distributed optimization setup, agents only have the information about their associated local functions, and the network of agents relies on communication between agents in order to find the solution to the global task presented in \eqref{globalfunction}. We now introduce some assumptions on the local and global functions.

\begin{assumption}\label{local_convex}
For any agent $i\in [n]$ in the network, we assume that the local cost function $f_i:\Xc \to\R$ is convex and differentiable.
\end{assumption}
From this assumption, we can immediately get that the global function $F$ is also convex and differentiable, but we impose an additional assumption on the global cost function as follows.
\begin{assumption}\label{global_assump}
The global function $F:\Xc \to\R$ is strongly convex. There exists a unique minimizer for $F$ and the optimal value denoted by $F^\star = F(x^\star)$ exists. The gradients of local functions $\nabla f_i(x)$ are locally continuously differentiable around $x^\star$.
\end{assumption} 
\subsection{Network Settings}
The agents form a network, modeled by a simple undirected graph $\Gc = (\Vc,\Ec)$, where the agents are denoted by nodes $\Vc = [n]$ and the connection between two agents $i$ and $j$ is captured by the edge $\{i,j\} \in \Ec$. The neighborhood of agent $i$ is denoted by $\Nc_i\triangleq\{j \in \Vc: \{i,j\}\in \Ec\}$. The agents work collectively to find the optimum of the global cost function, which is the sum of all local cost functions.  
\begin{assumption}\label{network_assump}
The graph $\Gc$ is connected, i.e., there exists a path between any two distinct agents $i,j \in \Vc$. The graph Laplacian is denoted by $\Lc \in \R^{n\times n}$.
\end{assumption}
The connectivity assumption implies that $\Lc$ has a unique null eigenvalue. That is, $\Lc \1_n=0$, and $\1_n$ is the only direction (eigenvector) recovering the zero eigenvalue. 

\subsection{Mirror Descent}
\label{subsec:MD}
We now provide a brief introduction of centralized mirror descent algorithm and explain the transition from discrete mirror descent (as mentioned in  \cite{nemirovsky1983originalmd} ) to a continuous-time setup. Later, in Section \ref{subsec:DMD} we derive the distributed mirror descent updates in continuous time. 

In gradient descent method, each iterate can be seen as an optimization problem on a simplified model, constructed by a first order approximation of a function plus a Euclidean regularizer. Mirror descent replaces the Euclidean regularizer with {\it Bregman divergence}. Bregman divergence is defined with respect to a distance generating function (DGF) $\phi:\Xc \to \R$, as follows
\begin{align}
    \Dc_{\phi} (x, x') \triangleq \phi(x) -\phi(x') - \langle \nabla\phi(x'), x-x'\rangle.
\end{align}
In discrete time, the mirror descent algorithm with learning rate $\eta$ is written as 
\begin{equation}\label{originalmd}
    \begin{aligned}
        x^{(k+1)} &= \underset{x\in \Xc}{\argmin} \bigg\{ F(x^{(k)}) + \eta \nabla F(x^{(k)})^\top (x - x^{(k)}) + \Dc_\phi(x, x^{(k)})\bigg\}.
    \end{aligned}
\end{equation}
Bregman divergence is regarded as a more general version of the Euclidean regularizer. When using the Euclidean distance as the Bregman divergence (i.e., $\Dc_\phi(x, x^{(k)})=\frac{1}{2}\|x-x^{(k)}\|^2$) we recover gradient descent. Hence, mirror descent is seen as a more general version of gradient descent.

\begin{assumption}\label{Assumption_phi}
 The distance generating function $\phi$ is closed, differentiable and $\mu_\phi$-strongly convex.
\end{assumption}

\begin{assumption}\label{Assumption_phi_local}
The Hessian of distance generating function, $\nabla^2 \phi$, is locally continuously differentiable around the neighborhood of $x^\star$ (the minimizer of F).
\end{assumption}
The two assumptions above on $\phi$ are satisfied by some of the commonly used Bregman divergences, such as $\phi(x) = \frac{1}{2}\norm{x}^2$, DGF of the Euclidean distance, and the negative entropy function $\phi(x) =  \sum_{j=1}^d [x]_j \log([x]_j)$, DGF of the Kullback–Leibler divergence.

Now, we introduce an equivalent form of the update above for more convenient analysis. This equivalent form is based on the {\it convex conjugate} (also known as {\it Fenchel dual}) of function $\phi$, which is denoted by $\phi^\star$ and defined as follows  
$$\phi^\star(z) \triangleq \underset{x \in \Xc }{\text{sup}}\{\langle x, z\rangle - \phi(x)\}.$$
From the definition, we can derive the the subsequent relationship, 
$$
z = \nabla\phi(x) \Longleftrightarrow x = \nabla\phi^\star(z).
$$
This means $\nabla \phi^\star$ will map the range of $z$ back to $\Xc$.
Assumption \ref{Assumption_phi} on the DGF $\phi$ guarantees the $\mu_\phi^{-1}$-smoothness property on $\phi^\star$ (see e.g.,
\cite{hiriart2012fundamentals}). Using the definition of $\phi^\star$, the update \eqref{originalmd} can be rewritten in the following equivalent form
\begin{equation} \label{step_by_step_md}
    \begin{aligned}
        % z^{(k)} &= \nabla\phi(x^{(k)})\\
        z^{(k+1)} &= z^{(k)} - \eta\nabla F(x^{(k)})\\
        x^{(k+1)} &= \nabla\phi^\star(z^{(k+1)}).
        % &\qquad+ \Dc_\phi(x, x^{(k)})\},
    \end{aligned}
\end{equation}
Then, the continuous-time update can be obtained by setting $\eta$ infinitesimally small as follows 
\begin{equation}
   \begin{aligned}
   \label{CT_CMD}
    \Dot{z} &= -\nabla F(x),\\
    x &=\nabla\phi^\star(z),\\
    x(0) = x_0, z(0)&=z_0 \:\:\text{with}\:\: x_0 =\nabla\phi^\star(z_0),
\end{aligned} 
\end{equation}
This setup was studied in \cite{krichene2015acceleratedmd}.
\subsection{Distributed Mirror Descent with Integral Feedback}
\label{subsec:DMD}
In this section, we introduce the distributed algorithm for mirror descent shown in \eqref{CT_CMD}. Our end goal is to have all agents converge to the global optimum in \eqref{globalfunction} and reach consensus. Motivated by \cite{cortes2013weightbalanceddigraph, kia2015distributed}, we use {\it integral feedback} to get
\begin{equation}
\label{setup}
\begin{aligned}
    \dot{z_i} &= - \nabla f_i(x_i) + \sum_{j \in \Nc_i} (x_j - x_i) + \int_0^t \sum_{j \in \Nc_i} (x_j - x_i)\\
     \vphantom{\int_0^t \sum_{j \in \Nc_i}}    x_i & =\nabla\phi^\star(z_i),\\
    % \vphantom{\int_0^t \sum_{j \in \Nc_i}}    
    &\text{with}~~~~x_i(0) = x_{i0},~~~~z_i(0) = z_{i0}, ~~~~~~\text{and}~~~~x_{i0} =\nabla\phi^\star(z_{i0}).
\end{aligned}
\end{equation}
The algorithm only utilizes gradient information of the local costs. The first equation updates the dual variable $z_i$ using gradient information, a consensus term, and the integral feedback. Then, the second equation updates the primal variable by mirroring the dual variable back with function $\phi^\star$. For convenience, we stack vectors from all agents and define the following notation,
\begin{equation}
\label{rename}
\begin{aligned}
    \Lb & \triangleq \Lc \otimes I_d\\
    \xb &\triangleq \col\{x_1, x_2,\ldots,x_n\}\\
    \zb & \triangleq \col\{z_1, z_2,\ldots,z_n\},\\
    \nabla\phi^\star(\zb) & \triangleq \col\{\nabla\phi^\star(z_1),\nabla\phi^\star(z_2),\ldots,\nabla\phi^\star(z_n)\}\\
    \nabla f(\xb) & \triangleq \col\{\nabla f_1(x_1), \nabla f_2(x_2),\ldots, \nabla f_n(x_n)\}.
\end{aligned}
\end{equation}
Additionally, we introduce a variable $\yb$ to replace the integral. Then, the dynamical system \eqref{setup} can be written using the newly defined notations,
\begin{equation}
    \begin{aligned}
    \Dot{\mathbf{z}} &= -(\nabla f(\mathbf{x}) + \mathbf{L} \mathbf{x} +   \mathbf{y}), \label{setup_with_y}\\
    \Dot{\mathbf{y}} &=  \mathbf{L}\mathbf{x},\\
    \xb &= \nabla\phi^\star(\zb),
\end{aligned}
\end{equation}
where $\mathbf{y}\in \mathbb{R}^{nd}$ and $\mathbf{y}(0)=\0.$

\section{Main Results}
In this section, we provide the convergence results of \eqref{setup_with_y}. In particular, we prove that under our assumptions, all agents in the network will converge exponentially fast to the global minimum of $F$ in \eqref{globalfunction}. In a previous work, the authors showed that under a subset of assumptions, the algorithm will asymptotically converge to the global optimum (without providing the rate).
\begin{theorem}\label{theorem1}[\cite{sun2020distributed}] 
    Given Assumptions \ref{local_convex}-\ref{Assumption_phi}, for any starting point $x_i(0) = x_{i0}, z_i(0) = z_{i0}$ with $x_{i0}=\nabla\phi^\star(z_{i0})$, the distributed mirror descent algorithm with integral feedback proposed in \eqref{setup} will converge to the global optimum asymptotically, i.e., $\lim_{t\to\infty} x_i(t)=x^\star$ for any $i\in [n]$.
\end{theorem}
The proof of this theorem can be found in \cite{sun2020distributed}, where it is also shown that agents reach consensus at the global optimal point, which is the unique equilibrium of the dynamical system \eqref{setup_with_y}. The equilibrium point for $\xb, \yb, \zb$ is denoted by 
$$\xb^\star = \1_n \otimes x^\star, \quad \yb^\star = -\nabla f(\xb^\star), \quad \zb^\star = \1_n \otimes z^\star = \1_n \otimes \nabla \phi(x^\star).$$

% $, \yb^\star, \zb^\star$, respectively, it can also be shown that $\yb^\star = -\nabla f(\xb^\star)$
\subsection{Coordinate Transformation}
We use the change of variables in \cite{sun2020distributed} for further analysis. Let $\Sb = \Lb^{\frac{1}{2}}$, and recall that $\Lb = \Lc \otimes I_d$ is a symmetric positive semi-definite matrix. We then introduce a new variable $\wb(t)=\Sb\int_0^t\xb(\tau)d\tau$. From \eqref{setup_with_y} it is easy to show that $\yb = \Sb \wb$. We then center the variables by moving the system's equilibrium to the origin as follows
\begin{equation}
    \begin{aligned}
        \Tilde{\xb} \triangleq \xb - \xb^\star, \quad
        \Tilde{\yb} \triangleq \yb - \yb^\star,\quad
        \Tilde{\wb} \triangleq \wb - \wb^\star,\quad
        \Tilde{\zb} \triangleq \zb - \zb^\star.
    \end{aligned}
\end{equation}
The first two equations in \eqref{setup_with_y} can be rewritten as 
\begin{equation}
\label{setup_recentered}
    \begin{aligned}
    \Dot{\Tilde{\zb}} &=  -(\nabla f(\Tilde{\xb} + \xb^\star) - \nabla f(\xb^\star)) - \mathbf{L} \Tilde{\xb}- \mathbf{S}\Tilde{\wb} ,\\
    \Dot{\Tilde{\wb}} &=  \Sb \Tilde{\xb}, 
\end{aligned}
\end{equation}
Next, we perform a dimension reduction on variable $\Tilde{\wb}$. Define $r \triangleq \frac{1}{\sqrt{n}}\1_n$ and let $\Lc = Q \Lambda Q^\top$, where $\Lambda = diag\{0, \lambda_1, ..., \lambda_{n-1}\}$. From Assumption \ref{network_assump} it is clear that $r$ is the first column of $Q$. We then define $R \in \R^{n\times(n-1)}$ such that $Q = [r,R]$. The following relationships follow subsequently
\begin{equation}
    r^\top R = 0, \quad R^\top R = I_{n-1}, \quad RR^\top = I_n - \frac{1}{n} \1_n \1_n^\top, \quad r^\top \Lc r = 0, \quad R^\top \Lc R \succ 0.
\end{equation}
Now, let
\begin{equation}
    \rb \triangleq r\otimes I_d,\quad \Rb \triangleq R\otimes I_d, \quad \Qb \triangleq Q \otimes I_d,
\end{equation}
and define new vectors by the following transformations from $\Tilde{\wb}$, 
$$\Wb \triangleq \Qb^\top \Tilde{\wb} = \begin{bmatrix}\rb^\top \\ \Rb^\top \end{bmatrix} \Tilde{\wb} = \begin{bmatrix}\rb^\top \Tilde{\wb}\\ \Rb^\top \Tilde{\wb}\end{bmatrix} = \begin{bmatrix}\Wb_1\\ \Wb_2\end{bmatrix}.$$
Note that for $\Wb_1$, from \eqref{setup_recentered} we can derive that
$$\dot{\Wb}_1 = \rb^\top \dot{\Tilde{\wb}} = \rb^\top \Sb \Tilde{\xb} = 0, \quad \Wb_1(0) = -\rb^\top \wb^\star = 0.$$
%the latter equation is true because the system equilibrium point is consensus at global optimum, $(\1_n \otimes I_d)\nabla f(\xb^\star) = \nabla F(x^\star) = 0$. 
Therefore $\Wb_1 \equiv 0$ for all time $t$, and the system variable can be represented by $\Wb_2$ only. $\Tilde{\wb} = \begin{bmatrix}\rb& \Rb \end{bmatrix} \Wb = \rb \Wb_1 + \Rb \Wb_2 = \Rb \Wb_2$.

Furthermore, we replace variable $\zb$ with $\xb$. Since $\zb = \nabla \phi(\xb)$, we have
$$\dot{\Tilde{\zb}} = \frac{d}{dt}(\zb - \zb^\star) = \nabla^2 \phi(\xb) \dot{\Tilde{\xb}}.$$
Assumption \ref{Assumption_phi} implies that $\nabla^2 \phi(\xb)$ is positive definite and therefore invertible.
Now, we can rewrite the system in \eqref{setup_recentered} using only variables $\Tilde{\xb}$ and $\Wb_2$ as follows
\begin{equation}
\label{setup_reduce_dim}
    \begin{aligned}
    \Dot{\Tilde{\xb}} &=  - \nabla^2 \phi(\Tilde{\xb} + \xb^\star)^{-1}(\nabla f(\Tilde{\xb} + \xb^\star) - \nabla f(\xb^\star) + \mathbf{L} \Tilde{\xb}+ \mathbf{S}\Rb\Wb_2) ,\\
    \Dot{\Wb}_2 &=  \Rb^\top \Sb \Tilde{\xb}, 
\end{aligned}
\end{equation}
Thus, the (exponential) stability of \eqref{setup_with_y} can be analyzed using the (exponential) stability of \eqref{setup_reduce_dim}.

\subsection{Exponential Convergence}
With the system transformation in place, we can discuss the convergence and stability of distributed mirror descent (with integral feedback) in the following theorem.
\begin{theorem}\label{main:thm}{(Main Result)}
    Given Assumptions \ref{globalfunction}-\ref{Assumption_phi_local}, the origin is a locally exponentially stable equilibrium of \eqref{setup_with_y} and \eqref{setup_reduce_dim}.
\end{theorem}
\begin{proof}
If we linearize the system \eqref{setup_reduce_dim} at the origin, using the notation $\nabla^2 \phi(\Tilde{\xb} + \xb^\star)^{-1}|_{\Tilde{\xb} = 0} = \Db$, $\nabla^2 f(\Tilde{\xb} + \xb^\star)|_{\Tilde{\xb} = 0} = \Hb$, the linearized version of \eqref{setup_reduce_dim} is
\begin{equation}\label{setup_linear}
    \begin{bmatrix} \Dot{\Tilde{\xb}}\\\Dot{\Wb}_2\end{bmatrix} = -
    \Mb
    \begin{bmatrix} {\Tilde{\xb}}\\{\Wb}_2\end{bmatrix}, ~~~\text{where}~~~
    \Mb \triangleq \begin{bmatrix}
   \Db   (\Hb+ \Lb) & \Db \mathbf{S}\mathbf{R}\\
   -\mathbf{R}^\top   \mathbf{S} &0
    \end{bmatrix}.
\end{equation}
We denote by $\lambda_1, ..., \lambda_{(2n-1)d}$ the eigenvalues of the linearized system matrix $\Mb$ in \eqref{setup_linear}. Based on Lemma \ref{eigenval_proof}, $Re[\lambda_i]>0 $ for all eigenvalues. Lemma \ref{eigenval_proof} and its proof are provided later in the paper. Now, from Theorem 3.2 in \cite{khalil2014nonlinear}, since $Re[\lambda_i]>0 $, the equilibrium of system \eqref{setup_reduce_dim}, as well as the equilibrium of system \eqref{setup_with_y} given by Theorem \ref{theorem1}, are both locally exponentially stable. This means there exists $\delta >0$ such that for any $\norm{col\{\xb, \yb, \zb\} - col\{\xb^\star, \yb^\star, \zb^\star\}}\leq \delta,$ the system state variables converge to the equilibrium (global optimal solution) exponentially fast. 
    \end{proof}
Recall from Theorem \ref{theorem1} that the system \eqref{setup_with_y} also exhibits global asymptotic convergence to the equilibrium. Then, for any starting point for $col\{\xb, \yb, \zb\}$, the state variables can converge to a neighborhood of radius $\delta$ of equilibrium in a finite time $T(\delta)$. Combined with the exponential convergence rate within the ball, this means that (after a finite time), the system exhibits exponential convergence to the global optimal solution. 

We now provide the following two lemmas used in the proof of Theorem \ref{main:thm}.
\begin{lemma}\label{H+L}
Given Assumptions \ref{globalfunction}-\ref{network_assump}, the matrix $(\Hb +\Lb)$ is positive definite.
\end{lemma}
\begin{proof}
First, $(\Hb +\Lb)$ is symmetric since both $\Hb$ and $\Lb$ are symmetric. For any non-zero vector $v \in \R^{nd}$, from Assumptions \ref{globalfunction} and \ref{network_assump}, we know that $$v^\top\Hb v \geq 0, v^\top \Lb v \geq 0 \Rightarrow v^\top (\Hb + \Lb)v = v^\top\Hb v + v^\top \Lb v\geq 0.$$ Furthermore, Since $\1_n$ is the unique eigenvector of $\Lc$ recovering the null eigenvalue, when $v^\top \Lb v = 0$, $v$ must satisfy $v = \1_n \otimes u$ for some $u$. Then, Assumption \ref{global_assump} ensures $v^\top\Hb v = (\1_n \otimes u)^\top \Hb (\1_n \otimes u) = u^\top \nabla^2 F(x^\star) u>0$. This shows that the symmetric matrix $\Hb + \Lb$ is positive definite.
\end{proof}

\begin{lemma}\label{eigenval_proof} Given Assumptions \ref{globalfunction}-\ref{Assumption_phi_local}, 
$Re[\lambda_i]>0$ for all eigenvalues $\lambda_1, ..., \lambda_{(2n-1)d}$ of $\Mb = \begin{bmatrix}
   \Db   (\Hb+ \Lb) & \Db \mathbf{S}\mathbf{R}\\
   -\mathbf{R}^\top   \mathbf{S} &0
    \end{bmatrix}$ .
\end{lemma}
\begin{proof}
For any $i \in [(2n-1)d]$, $\lambda_i$ must be a solution to $det(\Mb - \lambda I_{(2n-1)d}) = 0.$
First, let us rule out the possibility of having $\lambda_i=0$.
\begin{equation}
    \begin{aligned}
    det(\Mb) &= det\Bigg(\begin{bmatrix}
   \Db   (\Hb+ \Lb) & \Db \mathbf{S}\mathbf{R}\\
   -\mathbf{R}^\top   \mathbf{S} &0
    \end{bmatrix}\Bigg)\\
    &= det\Bigg(\begin{bmatrix}
   \Db   (\Hb+ \Lb) & 0\\
   -\mathbf{R}^\top   \mathbf{S} &I_{(n-1)d}
    \end{bmatrix}\begin{bmatrix}
   I_{nd} & (\Db   (\Hb+ \Lb))^{-1}\Db \mathbf{S}\mathbf{R}\\
   0 &\mathbf{R}^\top   \mathbf{S} (\Db   (\Hb+ \Lb))^{-1}\Db \mathbf{S}\mathbf{R}
    \end{bmatrix}\Bigg)\\
        &= det(\Db   (\Hb+ \Lb))det(\mathbf{R}^\top   \mathbf{S} (\Db   (\Hb+ \Lb))^{-1}\Db \mathbf{S}\mathbf{R})\\
    &= det(\Db )det(  \Hb+ \Lb)det(\mathbf{R}^\top   \mathbf{S}   (\Hb+ \Lb)^{-1} \mathbf{S}\mathbf{R}).
    \end{aligned}
\end{equation}
Since $\Rb^\top \Sb \Sb \Rb = (R^\top\Lc R) \otimes I_d \succ 0$, the null space of $\Sb \Rb$ is 0. Then, $\mathbf{R}^\top   \mathbf{S}   (\Hb+ \Lb)^{-1} \mathbf{S}\mathbf{R}$ is positive definite since $\Hb+ \Lb$ is positive definite (Lemma \ref{H+L}). As a result, this confirms that $det(\Mb) > 0$, implying $\lambda_i \neq 0$ for all $i \in [(2n-1)d]$.

The next step is to look at the characteristic polynomial of $\Mb$, where we have
\begin{equation}
    \begin{aligned}
    0&=det(\Mb - \lambda I_{(2n-1)d}) \\
    &=det\Bigg(\begin{bmatrix}
   \Db   (\Hb+ \Lb) - \lambda I_{nd} & \Db \mathbf{S}\mathbf{R}\\
   -\mathbf{R}^\top   \mathbf{S} &-\lambda I_{(n-1)d}
    \end{bmatrix}\Bigg) \\
    &=det(-\lambda I_{(n-1)d})det( \Db   (\Hb+ \Lb) - \lambda I_{nd} - \Db \Sb \Rb ( \lambda I_{(n-1)d} )^{-1} \Rb^\top \Sb )\\
    &=det(\Db) det(  (\Hb+ \Lb) - \lambda \Db^{-1} - \frac{1}{\lambda} \Sb \Rb  \Rb^\top \Sb )\\
    &=det((\Hb+ \Lb) - \lambda \Db^{-1} - \frac{1}{\lambda} \Sb (I_{nd} - \rb \rb^\top) \Sb ) \\
    &=det((\Hb+ \Lb) - \lambda \Db^{-1} - \frac{1}{\lambda} \Lb ).
    \end{aligned}
\end{equation}
Observe that $(\Hb+ \Lb) - \lambda \Db^{-1} - \frac{1}{\lambda} \Lb $ is a symmetric matrix, and $det((\Hb+ \Lb) - \lambda \Db^{-1} - \frac{1}{\lambda} \Lb) = 0$ implies that there exists a non-zero vector $v \in \R^{nd}$ for any solution of $\lambda$ such that 
$$v^\top ((\Hb+ \Lb) - \lambda \Db^{-1} - \frac{1}{\lambda} \Lb)v = 0.$$
Since $\Lb$ is positive semi-definite, $\Db^{-1}$ and $(\Hb +\Lb)$ are positive definite, $v^\top \Lb v \geq 0, v^\top \Db^{-1} v > 0, v^\top (\Lb + \Hb) v > 0$. When $v^\top \Lb v = 0$, 
$$\lambda = \frac{ v^\top (\Lb + \Hb) v}{v^\top \Db^{-1} v} >0,$$
when $v^\top \Lb v > 0$, 
$$\lambda = \frac{v^\top (\Lb + \Hb) v \pm \sqrt{(v^\top (\Lb + \Hb) v)^2 - 4 v^\top \Lb v v^\top \Db^{-1} v}}{2v^\top \Db^{-1} v},$$
certifying that $Re[\lambda]>0$ in both cases.
\end{proof}

\section{Numerical Simulation}
In this section, we use a real data-set to show the linear convergence of the training loss in a regression problem. We  will investigate the performance of distributed mirror descent with and without integral feedback. We utilize Euler's discretization scheme on algorithm \eqref{setup_with_y}. The resulting discrete-time algorithm for distributed mirror descent with integral feedback is provided below.
\begin{equation}\label{discrete_final}
    \begin{aligned}
        {z_i}^{(k+1)}&=  {z_i}^{(k)} - \bigg( \nabla f_i({x_i}^{(k)}) + {y_i}^{(k)}\vphantom{\sum_{j \in \Nc_i}}  +\sum_{j \in \Nc_i} ({x_i}^{(k)} - {x_j}^{(k)})  \bigg)\Delta t,
     \\
   {y_i}^{(k+1)} &= {y_i}^{(k)} + \sum_{j \in \Nc_i} ({x_i}^{(k)} - {x_j}^{(k)}) \Delta t \vphantom{\bigg)},\\
   {x_i}^{(k+1)} &= \nabla\phi^\star({z_i}^{(k+1)}) \vphantom{\sum_{j \in \Nc_i}\bigg)}.
\end{aligned}
\end{equation}
Details of this discretization is omitted in this manuscript and has been provided in \cite{sun2020distributed}.

\noindent
\textbf{Distance Generating Function for MD:}
We use the \textit{Negative Entropy} as our distance generation function $\phi$, namely, 
$$        \phi(x) = \sum_{j=1}^d [x]_j \log([x]_j) \Longrightarrow
        [z]_i = [\nabla \phi(x)]_i = 1+ \log([x]_i).$$
Based on Section \ref{subsec:MD}, the corresponding convex conjugate function $\phi^\star$ can be written below,
$$ \phi^\star(z) = \sum_{j=1}^d e^{[z]_j -1} \Longrightarrow [x]_i = [\nabla \phi^\star(z)]_i = e^{[z]_i -1}.$$
The reason for our choice of DGF is that Kullback–Leibler divergence is one the most commonly used Bregman divergences other than Euclidean distance, which simply reduces the method to distributed gradient descent with integral feedback as in \cite{kia2015distributed}.

\noindent
\textbf{Network Structure:}
We consider a 10-agent cycle network; each agent is connected to its previous and next agent in the loop.

\noindent
\textbf{Data Set and Model:} We use the \textit{Wine Quality Data Set} in UCI ML repository \cite{cortez2009modeling}. This is a regression data-set with 11 continuous input variables. Each agent is assigned 400 data instances with no overlap. For agent $i\in [n]$, we denote the input data and output data as $A_i \in \R^{400\times 11}$ and $b_i \in \R^{400}$, respectively. The model is a linear regression where the loss function is defined as the quadratic error loss, $f_i(x) = \frac{1}{2}\norm{A_i x - b_i}^2$. We can verify that this setup satisfies Assumptions \ref{local_convex} and \ref{global_assump}. The global objective function $F(x) = \sum_{i \in [n]} \frac{1}{2}\norm{A_i x - b_i}^2 = \frac{1}{2}\norm{\Ab x - \bb}^2 $, where $\Ab, \bb$ are the stacked version of $A_i, b_i$, respectively. Moreover, we can calculate the closed form solution of the global problem, $x^\star = \Ab^\dagger \bb$, where $\Ab^\dagger$ denotes the pseudo-inverse of $\Ab$.

Note that the selected model is not necessarily optimal for test prediction accuracy, and the aim of this numerical simulation is to show the ability of our proposed algorithm to converge exponentially fast to the optimal loss on a given data set. Finding a better model to fit this data set is not the main focus of this work.

\noindent
\textbf{Performance:}
We provide the trajectory of our proposed algorithm and also a comparison between our work and prior works \cite{li2016distributed,shahrampour2017distributed} on distributed mirror descent without integral feedback. In particular, we once run the algorithm without intergal feedback using diminishing step-size $\frac{1}{\sqrt{k}}$ to ensure consensus, and once using a constant step-size in optimization, which is unable to reach optimal solution.

The plot of $F(x_1^{(k)}) - F(x^\star)$ is shown in Fig. \ref{comparison}, representing the convergence speed of the three algorithms. We can see that our proposed algorithm converges faster than diminishing step-size setup, while the constant step-size setup without integral feedback fails to converge. We plot $\log(F(x_1^{(k)}) - F(x^\star))$ in Fig. \ref{comparison} to further display the exponential convergence (i.e., linear in log-scale) speed of our proposed method.

\begin{figure}[h]
\centering
\makebox[\textwidth]{\includegraphics[width=1\textwidth]{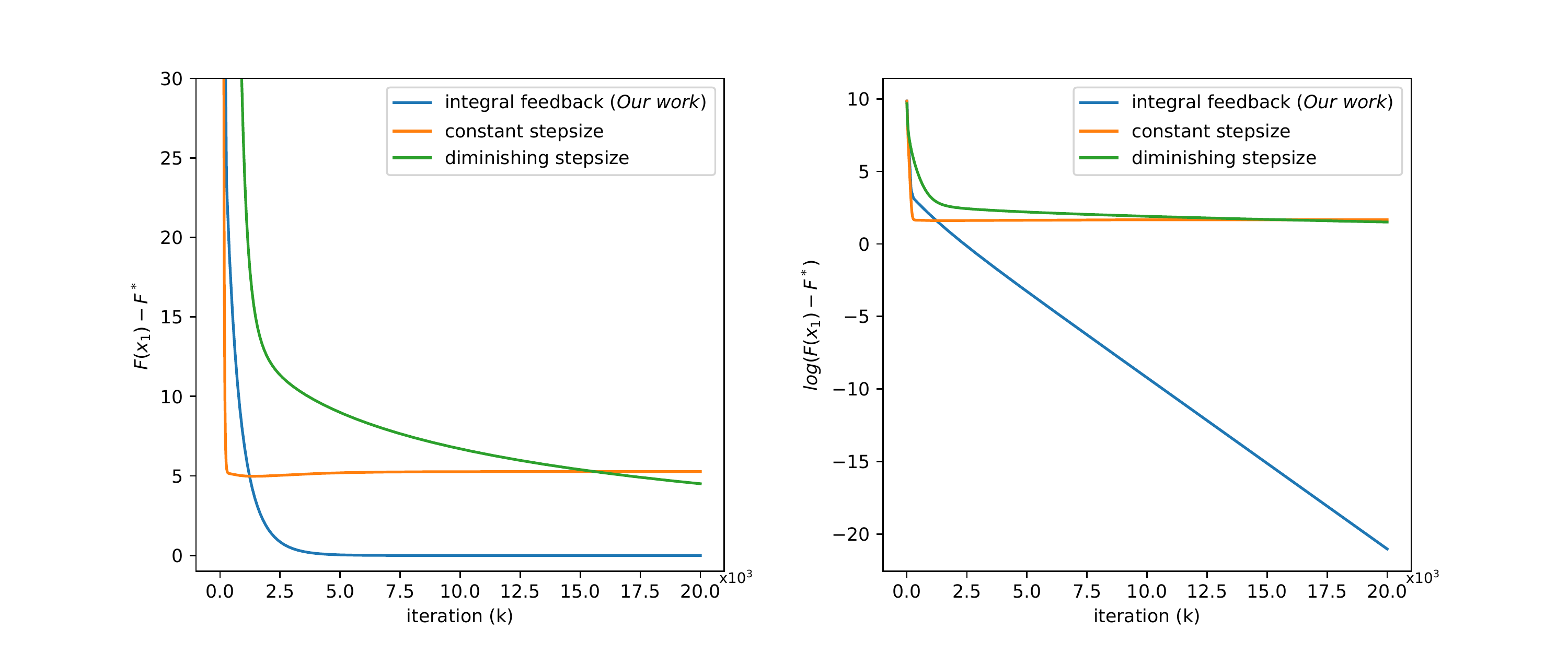}
}
\captionsetup{justification=centering,margin=1cm}

\caption{The trajectory of difference between $F(x)$ and optimal $F(x^\star)$ evaluated at agent $1$. \\Left: linear-scale, Right: log-scale}
\label{comparison}
\end{figure}

\section{Conclusion}
In this paper, we studied the distributed optimization problem, where a network of agents work together to find the optimal solution for a global objective function. We studied a distributed mirror descent algorithm that benefits from the idea of integral feedback. We established that the convergence rate of our algorithm is exponential (locally), which shows the advantage of adopting integral feedback for strongly convex problems. Our claim is supported by empirical results on a real data-set.

Though our work provides exponential convergence rate for strongly convex distributed optimization, more analysis is needed to generalize this work to other network settings, such as dynamic networks and networks with delays. Another interesting direction includes the theoretical analysis of the discretized version of this algorithm, shown in \eqref{discrete_final}. These are open questions left for future works.
% Acknowledgments---Will not appear in anonymized version

% \bibliography{yourbibfile}

\bibliography{References.bib}

\begin{thebibliography}{26}
\providecommand{\natexlab}[1]{#1}
\providecommand{\url}[1]{\texttt{#1}}
\expandafter\ifx\csname urlstyle\endcsname\relax
  \providecommand{\doi}[1]{doi: #1}\else
  \providecommand{\doi}{doi: \begingroup \urlstyle{rm}\Url}\fi

\bibitem[Ben-Tal et~al.(2001)Ben-Tal, Margalit, and Nemirovski]{ben2001ordered}
Aharon Ben-Tal, Tamar Margalit, and Arkadi Nemirovski.
\newblock The ordered subsets mirror descent optimization method with
  applications to tomography.
\newblock \emph{SIAM Journal on Optimization}, 12\penalty0 (1):\penalty0
  79--108, 2001.

\bibitem[Borovykh et~al.(2020)Borovykh, Kantas, Parpas, and
  Pavliotis]{borovykh2020interact}
Anastasia Borovykh, Nikolas Kantas, Panos Parpas, and Grigorios~A Pavliotis.
\newblock To interact or not? the convergence properties of interacting
  stochastic mirror descent.
\newblock In \emph{International Conference on Machine Learning (ICML) Workshop
  on ‘Beyond First order methods in ML Systems}, 2020.

\bibitem[Chavez et~al.(1997)Chavez, Moukas, and Maes]{chavez1997challenger}
Anthony Chavez, Alexandros Moukas, and Pattie Maes.
\newblock Challenger: A multi-agent system for distributed resource allocation.
\newblock In \emph{Proceedings of the first international conference on
  Autonomous agents}, pages 323--331, 1997.

\bibitem[Cortez et~al.(2009)Cortez, Cerdeira, Almeida, Matos, and
  Reis]{cortez2009modeling}
Paulo Cortez, Ant{\'o}nio Cerdeira, Fernando Almeida, Telmo Matos, and Jos{\'e}
  Reis.
\newblock Modeling wine preferences by data mining from physicochemical
  properties.
\newblock \emph{Decision Support Systems}, 47\penalty0 (4):\penalty0 547--553,
  2009.

\bibitem[{Doan} et~al.(2019){Doan}, {Bose}, {Nguyen}, and {Beck}]{8409957}
T.~T. {Doan}, S.~{Bose}, D.~H. {Nguyen}, and C.~L. {Beck}.
\newblock Convergence of the iterates in mirror descent methods.
\newblock \emph{IEEE Control Systems Letters}, 3\penalty0 (1):\penalty0
  114--119, 2019.

\bibitem[Gharesifard and Cort{\'e}s(2013)]{cortes2013weightbalanceddigraph}
Bahman Gharesifard and Jorge Cort{\'e}s.
\newblock Distributed continuous-time convex optimization on weight-balanced
  digraphs.
\newblock \emph{IEEE Transactions on Automatic Control}, 59\penalty0
  (3):\penalty0 781--786, 2013.

\bibitem[Hiriart-Urruty and Lemar{\'e}chal(2012)]{hiriart2012fundamentals}
Jean-Baptiste Hiriart-Urruty and Claude Lemar{\'e}chal.
\newblock \emph{Fundamentals of convex analysis}.
\newblock Springer Science \& Business Media, 2012.

\bibitem[Khalil(2014)]{khalil2014nonlinear}
Hassan~K Khalil.
\newblock \emph{Nonlinear control}.
\newblock Pearson Higher Ed, 2014.

\bibitem[Khan et~al.(2009)Khan, Kar, and Moura]{khan2009distributed}
Usman~A Khan, Soummya Kar, and Jos{\'e}~MF Moura.
\newblock Distributed sensor localization in random environments using minimal
  number of anchor nodes.
\newblock \emph{IEEE Transactions on Signal Processing}, 57\penalty0
  (5):\penalty0 2000--2016, 2009.

\bibitem[Kia et~al.(2015)Kia, Cort{\'e}s, and
  Mart{\'\i}nez]{kia2015distributed}
Solmaz~S Kia, Jorge Cort{\'e}s, and Sonia Mart{\'\i}nez.
\newblock Distributed convex optimization via continuous-time coordination
  algorithms with discrete-time communication.
\newblock \emph{Automatica}, 55:\penalty0 254--264, 2015.

\bibitem[Krichene et~al.(2015)Krichene, Bayen, and
  Bartlett]{krichene2015acceleratedmd}
Walid Krichene, Alexandre Bayen, and Peter~L Bartlett.
\newblock Accelerated mirror descent in continuous and discrete time.
\newblock In \emph{Advances in Neural Information Processing Systems
  (NeurIPS)}, pages 2845--2853, 2015.

\bibitem[Li et~al.(2016)Li, Chen, Dong, and Wu]{li2016distributed}
Jueyou Li, Guo Chen, Zhaoyang Dong, and Zhiyou Wu.
\newblock Distributed mirror descent method for multi-agent optimization with
  delay.
\newblock \emph{Neurocomputing}, 177:\penalty0 643--650, 2016.

\bibitem[Nedic and Ozdaglar(2009)]{nedic2009distributed}
Angelia Nedic and Asuman Ozdaglar.
\newblock Distributed subgradient methods for multi-agent optimization.
\newblock \emph{IEEE Transactions on Automatic Control}, 54\penalty0
  (1):\penalty0 48--61, 2009.

\bibitem[Nemirovsky and Yudin(1983)]{nemirovsky1983originalmd}
Arkadi{\u\i}~Semenovich Nemirovsky and David~Borisovich Yudin.
\newblock Problem complexity and method efficiency in optimization.
\newblock 1983.

\bibitem[Qu(2009)]{qu2009cooperativecontrol}
Zhihua Qu.
\newblock \emph{Cooperative control of dynamical systems: applications to
  autonomous vehicles}.
\newblock Springer Science \& Business Media, 2009.

\bibitem[{Rabbat}(2015)]{7383850}
M.~{Rabbat}.
\newblock Multi-agent mirror descent for decentralized stochastic optimization.
\newblock In \emph{IEEE 6th International Workshop on Computational Advances in
  Multi-Sensor Adaptive Processing (CAMSAP)}, pages 517--520, 2015.

\bibitem[Raginsky and Bouvrie(2012)]{raginsky2012Variancereduction}
Maxim Raginsky and Jake Bouvrie.
\newblock Continuous-time stochastic mirror descent on a network: Variance
  reduction, consensus, convergence.
\newblock In \emph{IEEE Conference on Decision and Control (CDC)}, pages
  6793--6800, 2012.

\bibitem[Shahrampour and Jadbabaie(2018)]{shahrampour2017distributed}
Shahin Shahrampour and Ali Jadbabaie.
\newblock Distributed online optimization in dynamic environments using mirror
  descent.
\newblock \emph{IEEE Transactions on Automatic Control}, 63\penalty0
  (3):\penalty0 714--725, 2018.

\bibitem[Shahrampour et~al.(2015)Shahrampour, Rakhlin, and
  Jadbabaie]{shahrampour2015distributed}
Shahin Shahrampour, Alexander Rakhlin, and Ali Jadbabaie.
\newblock Distributed detection: Finite-time analysis and impact of network
  topology.
\newblock \emph{IEEE Transactions on Automatic Control}, 61\penalty0
  (11):\penalty0 3256--3268, 2015.

\bibitem[Shi et~al.(2015)Shi, Ling, Wu, and Yin]{shi2015extra}
Wei Shi, Qing Ling, Gang Wu, and Wotao Yin.
\newblock Extra: An exact first-order algorithm for decentralized consensus
  optimization.
\newblock \emph{SIAM Journal on Optimization}, 25\penalty0 (2):\penalty0
  944--966, 2015.

\bibitem[Sun and Shahrampour(2020)]{sun2020distributed}
Youbang Sun and Shahin Shahrampour.
\newblock Distributed mirror descent with integral feedback: Asymptotic
  convergence analysis of continuous-time dynamics.
\newblock \emph{arXiv preprint arXiv:2009.06747}, 2020.

\bibitem[Xi and Khan(2017)]{xi2017dextra}
Chenguang Xi and Usman~A Khan.
\newblock Dextra: A fast algorithm for optimization over directed graphs.
\newblock \emph{IEEE Transactions on Automatic Control}, 62\penalty0
  (10):\penalty0 4980--4993, 2017.

\bibitem[Yang et~al.(2016)Yang, Liu, and Wang]{yang2016multi}
Shaofu Yang, Qingshan Liu, and Jun Wang.
\newblock A multi-agent system with a proportional-integral protocol for
  distributed constrained optimization.
\newblock \emph{IEEE Transactions on Automatic Control}, 62\penalty0
  (7):\penalty0 3461--3467, 2016.

\bibitem[Yu and A{\c{c}}{\i}kme{\c{s}}e(2020)]{yu2020rlc}
Yue Yu and Beh{\c{c}}et A{\c{c}}{\i}kme{\c{s}}e.
\newblock {RLC} circuits-based distributed mirror descent method.
\newblock \emph{IEEE Control Systems Letters}, 4\penalty0 (3):\penalty0
  548--553, 2020.

\bibitem[Yuan et~al.(2018)Yuan, Hong, Ho, and Jiang]{yuan2018optimal}
Deming Yuan, Yiguang Hong, Daniel~WC Ho, and Guoping Jiang.
\newblock Optimal distributed stochastic mirror descent for strongly convex
  optimization.
\newblock \emph{Automatica}, 90:\penalty0 196--203, 2018.

\bibitem[{Zeng} et~al.(2017){Zeng}, {Yi}, and {Hong}]{7744584}
X.~{Zeng}, P.~{Yi}, and Y.~{Hong}.
\newblock Distributed continuous-time algorithm for constrained convex
  optimizations via nonsmooth analysis approach.
\newblock \emph{IEEE Transactions on Automatic Control}, 62\penalty0
  (10):\penalty0 5227--5233, 2017.

\end{thebibliography}
\end{document}